\documentclass[11pt]{amsart}
\usepackage{verbatim, latexsym, amssymb, amsmath }
\def\R{\mathbb R}
\def\H{\mathbb H}
\def\N{\mathbb N}

\def\h2{\mathrm{area}}

\def\ar{|\Sigma_0|}
\def\a_t{\left\lvert\Sigma_t\right\rvert}
\def\b_i{\left\lvert\Sigma_i\right\rvert}
\def\d{\mathrm{div}}
\newcommand{\fint}{\mathop{\int\makebox(-15,2){\rule[4pt]{.9em}{0.6pt}}\kern-4pt}\nolimits}
\def\tf {\mathring{A}}
\def\n{\lvert\mathring{A}\rvert}

\def\dr{\lvert\partial_r^{\top}\rvert}
\def\dw_i{\lvert\widehat{\nabla} w_t\rvert}

\def\ah {asymptotically hyperbolic }
\def\imc{inverse mean curvature }

\def\tr{\mathrm{tr}}

\newtheorem*{conj}{Conjecture}
\newtheorem{thm}{Theorem}[section]
\newtheorem{lemm}[thm]{Lemma}

\newtheorem{prop}[thm]{Proposition}
\theoremstyle{remark}

\theoremstyle{definition}
\newtheorem{defi}[thm]{Definition}

\title{Insufficient convergence of inverse mean curvature flow on asymptotically hyperbolic manifolds}

\author{{Andr\'e Neves} ${}^{\dagger}$}
\email{aneves@math.princeton.edu}
\address{Fine Hall, Princeton University,
Princeton, NJ 08544, USA}

\thanks{\quad\ ${}^{\dagger}$\ The author was partially supported by NSF grant DMS-06-04164.}

\begin{document}

\maketitle \markboth{Insufficient convergence of inverse mean curvature flow on ...} { Andr\'e Neves} \maketitle

\begin{abstract}
We construct a solution to inverse mean curvature flow  on an asymptotically hyperbolic $3$-manifold which does not have the  convergence properties  needed in order to prove a Penrose--type inequality.  This contrasts sharply with the asymptotically flat case. The main idea consists in combining inverse mean curvature flow with work done by Shi--Tam regarding boundary behavior of compact manifolds. Assuming  the Penrose inequality holds, we also derive a nontrivial inequality for functions on $S^2$.
\end{abstract}

\section{Introduction}
A Penrose inequality for asymptotically flat $3$-manifolds was proven independently by Huisken-Ilmanen \cite{Huisken2}, using inverse mean curvature flow, and  Hugh Bray \cite{bray}, using a conformal deformation of the ambient metric. Recently, Hugh Bray and Dan Lee \cite{bray} extended Bray's approach  and prove a Penrose inequality for dimensions less than 8.
 Arguing by analogy with the asymptotically flat case, the following conjecture was stated by Xiadong Wang in \cite{wang}. Let $(M,g)$ is an asymptotically hyperbolic $3$-manifold with $R\geq-6$ and mass $M$ (see next subsection for definitions). 

\begin{conj}   If $\Sigma_0$ is an outermost sphere with $H(\Sigma_0)=2$, then
$$M\geq \left(\frac{|\Sigma_0|}{16\pi}\right)^{1/2}.$$
If equality holds then $(M,g)$ is isometric to an Anti--de Sitter--Schwarzschild manifold outside $\Sigma_0$.
\end{conj}

The main purpose of this paper is to show that, contrarily to what was suggested in \cite{wang}, the inverse mean curvature flow does not have the necessary convergence properties in order to prove this conjecture. Before we explain the main theorem of this paper we need to introduce some notation first.

\subsection{Notation and Definitions}\label{notation}

Given a complete noncompact Riemannian $3$-manifold $(M,g)$, we denote its connection by $D$, the Ricci curvature by $Rc$, and the scalar curvature by $R$.  The induced connection on a surface $\Sigma\subset M$ is denoted by $\nabla$, the exterior unit normal by $\nu$ (whenever its defined), the mean curvature by $H$ and the surface area by $|\Sigma|.$  

A sphere $\Sigma\subset M$ with mean curvature $H(\Sigma)=2$ is said to be  {\em outermost} if it is the boundary of a compact set and its outside region contains no other spheres with $H=2$. We say that $\Sigma$ is {\em outer minimizing} if every compact perturbation lying outside of $\Sigma$ has bigger surface area.

In what follows $g_0$ denotes the standard metric on $S^2$.

\begin{defi}
A complete noncompact Riemannian $3$-manifold $(M,g)$ is said to be \ah if the following are true:
\begin{itemize}
\item[(i)] There is a compact set $K\subset\subset M$  such that $M\setminus K$ is diffeomorphic to $\R^3$ minus an open ball.

\item [(ii)] With respect to the spherical coordinates induced by the above diffeomorphism, the metric can be written as
$$g=dr^2+\sinh^2r \,g_0+h/(3\sinh r) + Q$$
where $h$ is a symmetric $2$-tensor on $S^2$ and 
$$\lvert Q\rvert+\lvert D Q\rvert +\lvert D^2 Q\rvert+\lvert D^3 Q\rvert\leq C\exp(-4r)$$
for some constant $C$.
\end{itemize}
\end{defi}
 For simplicity, the manifolds we consider have only one end. The above definition is stated differently from the one given in \cite{wang}  (see also \cite{herzlich}). Nonetheless, using a simple substitution of variable $$t=\ln\left(\frac{\sinh(r/2)}{\cosh(r/2)}\right),$$ they can be  seen to be equivalent.
 
 Note that a given coordinate system on $M\setminus K$ induces a radial function $r(x)$ on $M\setminus K$. With respect to this coordinate system, we define the {\em inner radius} and {\em outer radius} of a surface $\Sigma\subset M\setminus K$ to be
  $$\underline{r}=\sup\{r\,|\, B_r(0)\subset \Sigma\}\quad\mbox{and}\quad\overline{r}=\inf\{r \,|\, \Sigma\subset B_r(0)\}$$
respectively.
respectively. Furthermore, we denote the coordinate spheres induced by a coordinate system by
$$\{|x|=r\}:=\{x\in M\setminus K\,|\,r(x)=r\}$$
and the radial vector by $\partial_r$. We stress that the radial function $r(x)$ depends on the coordinate system chosen. If $\gamma$ is an isometry of $\H^3$, the radial function $s(x)$ induced by this new coordinate system is such that
$$|s(x)-r(x)|\leq C\quad\mbox{for all x }\in M\setminus K,$$
where $C$ depends only on the distance from $\gamma$ to the identity. We denote  by $\overline s$ and  $\underline s$ the correspondent quantities defined with respect to this new coordinate system. 


 The mass $M$ of an \ah  manifold $(M,g)$ with $R\geq -6$ is given by
 $$M=\frac{1}{16\pi}\left[\left(\int_{S^2}\tr_{g_0} hd\mu_0\right)^{2}-\sum_{i=1}^3\left(\int_{S^2}\tr_{g_0} h \,x_id\mu_0\right)^{2}\right]^{1/2},$$
 where $(x_1,x_2,x_3)$ are the standard coordinates on $S^2\subset \R^3$. This quantity is well defined (i.e. independent of the coordinate system chosen for $M\setminus K$) by \cite{wang} (see also \cite{herzlich}).

The Anti--de Sitter--Schwarzschild metric $(S^2\times [t_0, +\infty), g_m)$ is given by
$$g_m=\frac{dt^2}{1+t^2-m/t}+t^2g_0,$$
where we choose $t_0$  so that the mean curvature of the coordinate sphere $\Sigma_0=\{|x|=t_0\}$ is $2$. A change of variable (see \cite[page 294]{wang}) shows that the metric can be written as
$$g=dr^2+(\sinh^2 r+m/(3\sinh r)g_0+P,$$
where $P$ is  term with order $\exp(-5r)$.
An explicit computation reveals that the scalar curvature equals $-6$ and that 
$$M=\frac{m}{2}=\left(\frac{|\Sigma_0|}{16\pi}\right)^{1/2}.$$

\subsection{Statement of the main results}

We start by briefly describing how inverse mean curvature flow could prove the conjecture. Find a family of surfaces $(\Sigma_t)_{t\geq 0}$ with initial condition $\Sigma_0$  such that
$$\frac{dx}{dt}=\frac{\nu}{H(\Sigma_t)}.$$
Note that the existence theory  for a weak solution developed in \cite[Section 3]{Huisken2} can be used in the current setting. Moreover,  the same arguments in \cite[Section 5]{Huisken2} show that the quantity (called the {\em Hawking mass})
$$m_H(\Sigma_t):=\frac{ {|\Sigma_t|^{1/2}}}{(16\pi)^{3/2}}\left(16\pi-\int_{\Sigma_t}H^2-4\,d\mu_t\right)$$
is monotone nondecreasing along the flow. Therefore,
$$\left(\frac{|\Sigma_0|}{16\pi}\right)^{1/2}=m_H(\Sigma_0)\leq \lim_{t\to\infty}m_H(\Sigma_t).$$
The result would follow if one could show that the limit of the Hawking mass is not bigger than $M$.

 In the asymptotically flat case,  Huisken and Ilmanen \cite[Section 7]{Huisken2} showed this by proving that
$$\liminf_{t\to\infty}\frac{\mbox{area}(B_{\overline r_t}(0))}{\mbox{area}(B_{\underline r_t}(0))}=\liminf_{t\to\infty}\frac{\overline r_t}{{\underline r_t}}=1,$$
where $\overline r_t$ and $\underline r_t$ denote the outer radius and inner radius of $\Sigma_t$ respectively.  In our setting, it is not hard to see that in order for the  limit of the Hawking mass to be smaller than $M$ we  need to find an isometry $\gamma$ of $\H^3$ such that, with respect to the induced coordinate system, the following two properties hold:
\begin{itemize}
\item[1)]$$\liminf_{t\to\infty}\frac{|B_{\overline s_t}(0)|}{|B_{\underline s_t}(0)|}=\liminf_{t\to\infty}(\overline s_t-\underline s_t)=0,$$
where $\overline s_t$ and $\underline s_t$ denote, respectively, the outer radius and inner radius of $\Sigma_t$ with respect to the radial function $s(x)$ induced by $\gamma$;
\item[2)] If the metric with respect to the coordinates induced by $\gamma$ is written as
$$g=ds^2+\sinh^2s \,g_0+h^{\gamma}/(3\sinh s) + P,$$
then
\begin{equation*}
 \int_{S^2}x_i\tr_{g_0} h^{\gamma} d\mu_0=0\quad\mbox{for}\quad i=1,2,3,
\end{equation*}
where $x_i$ denotes the coordinate functions of the unit sphere in $\R^3$. 
\end{itemize}
If these properties do not hold,  it is impossible to compare the limit of the Hawking mass with the mass of the manifold. As a matter of fact, during the proof of the main theorem, we will construct   a solution to inverse mean curvature flow  for which the limit of the Hawking mass is bigger than the mass of the manifold. 

We can now state the main theorem.

\begin{thm}\label{main}There is an asymptotically hyperbolic $3$-manifold $(M,g)$ with scalar curvature $-6$ and for which its boundary $\Sigma_0$ is an outer-minimizing sphere with $H(\Sigma_0)=2$ satisfying  the following property. 

There is a smooth solution to inverse mean curvature flow $(\Sigma_t)_{t\geq 0}$ with initial condition $\Sigma_0$ such that {\em for every} coordinate system  we have
$$\liminf_{t\to\infty}(\overline s_t-\underline s_t)>0.$$
\end{thm}

In the next section we prove this theorem leaving all the technical aspects for the remaining sections. In that section, we also discuss whether the asymptotically hyperbolic manifold constructed in Theorem \ref{main} constitutes or not a counterexample to the Penrose inequality. In Section \ref{aha} some basic properties of spheres in asymptotically hyperbolic  manifolds are proven. In Section \ref{long} we prove a long time existence result for inverse mean curvature flow on Anti--de Sitter--Schwarzschild space. It is important that the estimates in  this section do not depend on the area of our initial condition and this requires a careful bookkeeping. Finally, in Section \ref{stf} we adapt the work of Shi-Tam \cite{shi-tam} and Mu-Tao-Yau \cite{mutao} to prove long time existence for a flow inspired in \cite{shi-tam}.

{\bf Acknowledgments } The author would like to express his thanks to Gang Tian for many useful discussions and also for his interest in this work.

\section{Proof of the main theorem}

We now prove the main theorem.

\begin{proof}[Proof of Theorem \ref{main}]
	Consider the ambient manifold to be Anti--de Sitter--Schwarzschild  $(S^2\times [t_0, +\infty), g_m)$  with positive mass.
	Set $f$ to be a  smooth function on $S^2$ with
	$$\fint_{S^2}\exp(2f)d\mu_0=1$$
	that is invariant under reflection on all coordinate planes and consider
	$$\Sigma(r_0)=\{( r_0+f(\theta),\theta)\,|\,\theta \in S^2\}\subset S^2\times [t_0, +\infty).$$
	According to Proposition \ref{nelsa} e) we know that
	$$\lim_{r_0\to\infty}m_{H}(\Sigma(r_0))=\frac{m}{2}\fint_{S^2}\exp(-f)d\mu_0>\frac{m}{2},$$
	where the last inequality is a consequence of H\"older's inequality. 
	
	Choose $r_0$ sufficiently large such that $\Sigma(r_0)$ satisfies hypothesis (H) of Section \ref{long} and
	$$m_{H}(\Sigma(r_0))>\frac{m}{2}.$$
	This is possible  because, due to 
 Proposition \ref{nelsa}, we  know that
	$$\lim_{r_0\to\infty} H=2\quad\mbox{and}\quad\lim_{r_0\to\infty}\n^2=0.$$ 
	Therefore, we can apply Theorem \ref{imcf} and conclude the existence of a smooth solution $(\Sigma_t)_{t\geq 0}$ to inverse mean curvature flow where, by monotonicity of Hawking mass,
	$$\frac{m}{2}<m_{H}(\Sigma(r_0))\leq\lim_{t\to\infty}m_{H}(\Sigma_t).$$
Denote the induced metric  on $\Sigma_t$ by $g_t$. The above inequality  implies
\begin{lemm}\label{mao} The Gaussian curvature $\widehat K_t$ of $\Sigma_t$ with respect to the normalized metric
$$\hat g_t:=(4\pi)|\Sigma_t|^{-1}g_{t}$$
does not converge to one when $t$ goes to infinity.
\end{lemm}	
\begin{proof}
From Theorem \ref{imcf} (iii) we know that	
$$\Sigma_t=\{(\hat r_t+f_t(\theta),\theta)\,|\,\theta \in S^2\},$$
where $\hat r_t$ is such that 
$$|\Sigma_t|=4\pi\sinh^2 \hat r_t$$
and the functions $f_t$ converge to a smooth function $f_{\infty}$ defined on $S^2$. Moreover, Proposition \ref{nelsa} (more precisely, identity \eqref{gibson}) implies that the metric $\hat g_t$ converges to $\hat g=\exp(2f_{\infty})g_0.$ Note that the ambient metric is preserved by reflections with respect to the coordinate planes and thus the metric $\hat g$ also shares these symmetries.

Suppose that $\widehat K_t$ converges to one. Then $\hat g$ is a constant scalar curvature metric which is symmetric  under reflection on the coordinate planes and so $f_{\infty}$ must be identically zero. If this were true, it would follow from Proposition \ref{nelsa} e) that
$$\lim_{t\to\infty}m_{H}(\Sigma_t)=\frac{m}{2}$$
and this is impossible.
\end{proof}

Outside $\Sigma(r_0)$, i.e., on the region $$N:=\bigcup_{t\geq 0}\Sigma_t,$$ the metric $g_m$ can be written as
$$g_m=\frac{dt^2}{H^2}+g_t.$$
We want to find a
new asymptotically hyperbolic metric $\bar g$ with $R(\bar g)=-6$ such that, with respect to this new metric, the mean curvature of
$\Sigma_0$ is $2$, $(\Sigma_t)_{t\geq 0}$ is a solution to \imc flow, and the induced metric on $\Sigma_t$ by $\bar g$ coincides with $g_t$. This would finish the proof for the following two reasons. 

First, because $(\Sigma_t)_{t\geq 0}$ is a smooth solution to inverse mean curvature flow for $\bar g$, there is a smooth function $u$ on $N$ such that
$$\mbox{div}_N\left(\frac{\bar \nabla u}{|\bar \nabla u|}\right)=0\quad\mbox{and}\quad u^{-1}(t)=\Sigma_t.$$ Therefore, if $\Sigma'$ is a surface in $N$ containing $\Sigma(r_0)$ in its interior,  the divergence theorem implies that
$$|\Sigma(r_0)|=\int_{\Sigma(r_0)}|\bar \nabla u|^{-1}\langle \nu,\bar \nabla u\rangle d\bar \mu=\int_{\Sigma'}|\bar \nabla u|^{-1}\langle \nu,\bar \nabla u\rangle d\bar \mu\leq |\Sigma'|$$
and thus $\Sigma(r_0)$ is outer-minimizing.

Second, the intrinsic geometry of $\Sigma_t$ is maintained and so we know from Lemma \ref{mao} that the Gaussian curvature of $\Sigma_t$ with respect to the normalized metric does not converge to one. Proposition \ref{nelsa} f) implies that no matter the coordinate system we choose we will always have
$$\liminf_{t\to\infty}(\overline s_t-\underline s_t)>0.$$

The construction of the metric $\bar g$ is inspired by the work of Shi and Tam \cite{shi-tam}. Consider smooth positive functions $u$ defined on $N$
such that $$u_{\Sigma_0}:=H(\Sigma_0)/2$$ and
    \begin{equation}\label{equacao}
    2H^2\frac{\partial u}{\partial t}=2u^2\Delta_t u+4u^2H\langle\nabla u,\nabla H^{-1}\rangle+
    (u-u^3)(R_t+6-2H\Delta_t H^{-1}),
    \end{equation}
where the Laplacian and gradient term are computed with respect to the metric $g_t$ and $R_t$ is the scalar
curvature of $\Sigma_t$. Having such a function $u$,  the new metric is defined to be
$$\bar g:=\frac{u^2}{H^2}dt^2+g_t$$
and it has scalar curvature $-6$ by Lemma \ref{yo}.
Note that the intrinsic geometry of $\Sigma_t$ is preserved and  the mean curvature and the exterior normal vector of $\Sigma_t$ computed with respect to $\bar g$ equal
    $$\bar H(\Sigma_t)=H(\Sigma_t)/u\quad\mbox{and}\quad\bar\nu=\nu/u$$
    respectively. Thus $$\frac{\bar \nu}{\bar H}=\frac{\nu}{H}$$
    and this implies that $(\Sigma_t)_{t\geq 0}$ is indeed a solution to \imc flow for the new metric with
    $\bar H(\Sigma_0)=2.$
    
 We are only left to check that equation \eqref{equacao} has a solution. Note that, provided we choose $r_0$ sufficiently large, Proposition \ref{nelsa}  and Theorem \ref{imcf} imply that
    $$R_t+6-2H\Delta_t H^{-1}>0$$
for all $t$. It is  important to remark that this estimate holds because the constants on Theorem \ref{imcf} do not depend on $\underline r_0$ but only on $\overline r_0-\underline r_0$. Therefore, Theorem \ref{shiflow} implies that equation \eqref{equacao} admits a solution and that the metric $\bar g$ is asymptotically hyperbolic.
	
\end{proof}

\subsection{A nontrivial consequence of the Penrose inequality} 

Assuming that the Penrose inequality holds as conjectured by Xiadong Wang, we will argue that for every smooth function $f$ defined on $S^2$ with
$$\fint_{S^2}\exp(2f)d\mu_0=1$$
we have
$$\left(\fint_{S^2} K_f\exp(3f)d\mu_0\right)^{2}-\sum_{i=1}^3\left(\fint_{S^2} K_f\exp(3f) \,x_id\mu_0\right)^{2}\geq 1,$$
where $K_f$ denotes the Gaussian curvature of $\exp(2f)g_0.$ A simple computation shows that an equality is attained if $\exp(2f)g_0$ has constant scalar curvature.

In what follows we use the same notation as in the proof of Theorem \ref{main}.  Set $f$ to be a  smooth function on $S^2$ with
	$$\fint_{S^2}\exp(2f)d\mu_0=1$$
	 and consider
	$$\Sigma(r_0)=\{( r_0+f(\theta),\theta)\,|\,\theta \in S^2\}\subset S^2\times [t_0, +\infty).$$
	Given a geometric quantity $T$ defined on $\Sigma(r_0)$, we use the notation $T=O(\exp(-\underline r_0)$ whenever we can find a constant $C$ for which
	$$|T|\leq C\exp(-\underline r_0.)$$
	
	Denote by $M(r_0)$ the mass of the metric $\bar g$ constructed in the proof of Theorem \ref{main}. It is not hard to see that, by choosing $r_0$ sufficiently large, we can have $f_{\infty}$ (defined in Theorem \ref{imcf} (iii)) and $w_{\infty}$ (defined in Theorem \ref{shiflow} (i)) respectively, as close to $f$ and $w_0$ as we want. 
	Moreover, from Proposition \ref{nelsa} d), we have that
	$$2w_0=|\Sigma_0|(H-2)/(4\pi)=\widehat K(\Sigma(r_0))+O(\exp(-\underline r_0)),$$
	where the Gaussian curvature is computed with respect to the normalized metric
	$\hat g(r_0):=4\pi|\Sigma(r_0)|^{-1}g_{\Sigma(r_0)}.$
	Therefore, denoting the mass two tensor of $\bar g$  by $\bar h$, we have that
	$$(16\pi)^{1/2}\bar h|\Sigma(r_0)|^{-1/2}$$
is well approximated by
	\begin{multline*}
	\left(m{16\pi}^{1/2} {|\Sigma(r_0)|}^{-1/2}+2\exp(3f)w_{0}\right)g_0\\
	=\left(\widehat K(\Sigma(r_0))\exp(3f)+O(\exp(-\underline r_0))\right)g_0.
	\end{multline*}
	Because the metric $\hat g(r_0)$ converges to $\exp(2f)g_0$ (Proposition \ref{nelsa} a)), we obtain that
	\begin{multline*}
	\lim_{r_0\to\infty} M(r_0)^2\frac{16\pi} {|\Sigma(r_0)|}\\
	=\left(\fint_{S^2} K_f\exp(3f)d\mu_0\right)^{2}-\sum_{i=1}^3\left(\fint_{S^2} K_f\exp(3f) \,x_id\mu_0\right)^{2}.
	\end{multline*}
	If we assume the Penrose inequality, we know that
	$$M(r_0)\left(\frac{16\pi} {|\Sigma(r_0)|}\right)^{1/2}\geq 1$$
	and so the desired inequality follows.

\section{Basic properties of graphical surfaces on asymptotically hyperbolic  $3$-manifolds}\label{aha}

In this section $(M,g)$ denotes an asymptotically hyperbolic manifold with some given coordinate system on $M-\setminus K$. Given a function $f$ on $S^2$ we consider the surfaces
$$\Sigma(q_0)=\{( q_0+f(\theta),\theta)\,|\,\theta \in S^2\}\subset M-\setminus K.$$
The function $f$ satisfies hypothesis $(I)$ if there are constants $V, V_0$ such that
\begin{equation*}
(I)\qquad\left \{ \begin{aligned}
				 &|f|\leq V,\\
				 &|\nabla_0 f|\leq V_0,				
				\end{aligned}
\right.
\end{equation*} where $\nabla_0$ denotes the connection with respect to the round metric on $S^2$.  We denote by $\overline s(q_0)$ and $\underline s (q_0)$, respectively, the outer radius and inner radius of $\Sigma(q_0)$, where $s(x)$ is the radial function induced by some coordinate system $\gamma$.

Given any geometric quantity $T$ defined on $\Sigma(q_0)$, we use the notation
$$T=O(\exp(-kr))$$
when we can find a constant $C=C(g,V,V_0)$ for which
$$|T|\leq C\exp(-kr).$$

The next proposition collects some properties for the surfaces $\Sigma(q_0)$ when $q_0$ is very large.

\begin{prop} \label{nelsa}Assume that $f$ satisfies hypothesis $(I)$.  The following properties hold:
\begin{enumerate}

	\item[a)]
	When $q_0$ goes to infinity, the normalized metrics 
	$$\hat g(q_0):=4\pi|\Sigma(q_0)|^{-1}g_{\Sigma(q_0)}$$
	converge to
	$$\hat g:=\left(\fint_{S^2}\exp(2f)d\mu_0\right)^{-1}\exp(2f)g_0.$$
	\item[b)]There is a constant $C=C(g,V,V_0)$ such that, for all $q_0\geq 1$,
$$ |\Sigma(q_0)||H-2|+|\Sigma(q_0)||\tf|\leq C+C\sup_{S^2}|\nabla^2_0 f|$$
and
$$\sup_{S^2}|\nabla^2_0 f| \leq C(|\Sigma(q_0)||H-2|+|\Sigma(q_0)||\tf|)+C;$$
\item[c)] Assume that $$\sup_{S^2}|\nabla_0^k f|\leq E\quad\mbox{for all }k=2,\cdots,n-1.$$ There is a constant $C=C(g,E,V,V_0)$ such that for all $q_0\geq 1$
$$|\Sigma(q_0)|^{n+2}|\nabla^n A|^2\leq C+C\sup_{S^2}|\nabla^n_0 f|^2$$
and
$$\sup_{S^2}|\nabla^n_0 f|^2\leq C+C|\Sigma(q_0)|^{n+2}|\nabla^n A|^2;$$
\item [d)] The mean curvature of $\Sigma(q_0)$ satisfies
$$H^2-4=4K(\Sigma)+2\n^2-\frac{2\tr_{g_0}h}{\sinh^3 r}+O(\exp(-4r));$$
\item [e)]
$$
\lim_{q_0\to\infty}m_{H}(\Sigma(q_0))=\frac{1}{4}\left(\fint_{S^2}\exp(2f)d\mu_0\right)^{1/2}\fint_{S^2}\tr_{g_0}h\exp(-f) d\mu_0.$$
\item [f)] There is a coordinate system $\gamma$ for which
$$\lim_{q_0\to\infty}(\overline s(q_0)-\underline s (q_0))=0$$
if and only if
$$\widehat K=\lim_{q_0\to\infty} \widehat K(q_0)=1,$$
where $\widehat K(q_0)$ is the Gaussian curvature of $\Sigma(q_0)$ with respect to $\hat g(q_0)$.
\end{enumerate}

\end{prop}
Note that this proposition holds, with obvious modifications, if 
$$\Sigma(q_0)=\{( q_0+f_{q_0}(\theta),\theta)\,|\,\theta \in S^2\},$$
where the functions $f_{q_0}$ converge to  a function $f$ on $S^2$.
\begin{proof}
Consider tangent vectors to $\Sigma(q_0)$ $$\partial_i:=\frac{\partial f}{
\partial {\theta_i}}\partial_r+{\partial_{\theta_i}},\quad i=1,2,$$
where $(\theta_1,\theta_2)$ represent coordinates on $S^2$ which are orthonormal (with respect to $g_0$) at a given point $p$. 

The induced metric on $\Sigma(q_0)$ is given by
$$g_{ij}=\frac{\partial f}{
\partial {\theta_i}}\frac{\partial f}{
\partial {\theta_j}}+\sinh^2 (q_0+f) g_0(\partial_{\theta_i},\partial_{\theta_j})+O(\exp(-r))
$$
and so
$$\sqrt {\mbox{det} g_{ij}}=\sinh^2(q_0+f)\sqrt{\mbox{det} g_0}+O(1)=\sinh^2(q_0)\exp(2f)\sqrt{\mbox{det} g_0}+O(1).$$
This implies that
\begin{equation}\label{gibson}
\lim_{q_0\to\infty}\frac{|\Sigma(q_0)|}{4\pi \sinh^2 q_0}=\fint_{S^2}\exp(2f)\,d\mu_0
\end{equation}
and the first property follows from the fact that
$$\lim_{q_0\to\infty} (\sinh q_0)^{-2} g_{ij}=\exp(2f)g_0(\partial_{\theta_i},\partial_{\theta_j}).$$

Denoting the connection with respect to the standard hyperbolic metric by  $\bar D$,
 we have
$$\bar D_{\partial_r} \partial_r=0,\quad \bar D_{{\partial_{\theta_j}}}\partial_r
=\frac{\cosh r}{\sinh r}{\partial_{\theta_j}},\quad \bar D_{\partial_{\theta_i}}{\partial_{\theta_j}}=-\sinh r {\cosh r}\delta_{ij}\partial_r,$$
and
$$|D-\bar D|\leq C\exp(-3r)$$
for some $C=C(g)$.

Thus,
\begin{multline*}
D_{\partial_i}\partial_j =\frac{\partial^2 f}{\partial {\theta_j}\partial \theta_i}\partial_r+\frac{\partial
f}{\partial {\theta_j}}\frac{\partial
f}{\partial {\theta_i}}D_{\partial_r}\partial_r+\frac{\partial
f}{\partial {\theta_j}}D_{\partial_{\theta_i}}\partial_r+\frac{\partial
f}{\partial {\theta_i}}D_{\partial_{r}}\partial_{\theta_j}+ D_{\partial_{\theta_i}}\partial_{\theta_j}\\
=-\cosh r \sinh r \delta_{ij}\partial_r+ \frac{\partial^2 f}{\partial {\theta_j}\partial \theta_i}\partial_r+O(1)\partial_{\theta_j}+O(\exp(-r)).
\end{multline*}
An easy computation shows that the exterior unit normal is given by
\begin{equation}\label{spring}
\nu=(1+O(\exp(-2r))\partial_r+O(\exp(-2r))\partial_{\theta_1}+O(\exp(-2r))\partial_{\theta_2}
\end{equation}
and thus
$$A_{ij}=2\frac{\cosh r}{\sinh r}g_{ij}-\frac{\partial^2 f}{\partial {\theta_j}\partial \theta_i}+O(1).$$
This implies Property b).

Property c) follows from what was done above plus some tedious computations. We now prove Property d).

It was shown in \cite[Lemma 3.1.]{neves2} that
$$Rc(\nu,\nu)+2=-\frac{\tr_{g_0}h}{2\sinh^3 r}+O(\exp(-4r))\quad\mbox{and}\quad R=-6+O(\exp(-4r)).$$

Combining this with Gauss equations we obtain that
\begin{align*}
	H^2-4&=4K(\Sigma(q_0))+2\n^2+4(R(\nu,\nu)-R/2-1)\\	
	&=4K(\Sigma(q_0))+2\n^2-\frac{2\tr_{g_0}h}{\sinh^3 r}+O(\exp(-4r)).
\end{align*}
 
Combining Property a) with Property d), it follows from the definition of Hawking mass that
\begin{align*}
\lim_{q_0\to\infty}m_{H}(\Sigma(q_0))&=\lim_{q_0\to\infty}\frac{|\Sigma(q_0)|^{1/2}}{(16\pi)^{3/2}}
\int_{\Sigma(q_0)}\frac{2\tr_{g_0}h}{\sinh^3
r}d\mu\\
&=\lim_{q_0\to\infty}\frac{1}{(16\pi)^{3/2}}\fint_{\Sigma(q_0)}\frac{2|\Sigma(q_0)|^{3/2}}{\sinh^3
r}\tr_{g_0}h d\hat\mu\\
 &=\frac{1}{4^{3/2}}\left(\fint_{S^2}\exp(2f)d\mu_0\right)^{3/2}\fint_{S^2}2\tr_{g_0}h\exp(-3f)
d\hat\mu\\
&=\frac{1}{4}\left(\fint_{S^2}\exp(2f)d\mu_0\right)^{1/2}\fint_{S^2}\tr_{g_0}h\exp(-f) d\mu_0.
\end{align*}

Finally, we prove Property e).  Given a coordinate system induced by an isometry $\gamma$ of $\H^3$, we consider the 
function on $\Sigma(q_0)$ given by
$$w(x)=s(x)-\hat q_0\quad\mbox{where}\quad |\Sigma(q_0)|=4\pi\sinh^2 q_0,$$
where $s(x)$ is the radial function for this coordinate system.
For all $q_0$ sufficiently large, $\Sigma(q_0)$ is graphical over the coordinate spheres for this new coordinate system and so
$$\limsup_{q_0\to\infty} \left(|\partial_s^{\top}|^2+(1-\langle\nu,\partial_s\rangle)\right)\exp(2q_0)<\infty.$$
Due to \cite[Propostion 3.3]{neves2}, we know that
\begin{multline*}
    \Delta s = (4-2\lvert\partial _s^{\top} \rvert^2 )\exp(-2s)+2-H\\ +(H-2)(1-\langle \partial _s, \nu \rangle)+(1-\langle \partial _s, \nu \rangle)^2
        + O(\exp(-3s)).
\end{multline*}
Therefore, Property d) implies that  $w$ satisfies the following equation with respect to $\hat g(q_0)$

\begin{equation}\label{coroa}
\hat \Delta w = \exp(-2w)-\widehat K(q_0)+P(q_0),
\end{equation}
where $$\lim_{q_0\to\infty}\fint_{\Sigma(q_0)}|P(q_0)|d\hat \mu=0.$$

Suppose the  coordinate system induced by $\gamma$ is such that
$$\lim_{q_0\to\infty}\overline s(q_0)-\underline s (q_0)=0.$$
Then
$$\lim_{q_0\to\infty} w=0$$
and so equation \eqref{coroa} implies that
$$\lim_{q_0\to\infty}\widehat K(q_0)=1.$$

Assume for simplicity that
$$\fint_{S^2}\exp(2f)\,d\mu_0=1$$
because, according to \eqref{gibson}, this implies that
$$\lim_{q_0\to\infty}\hat q_0-q_0=0.$$
  If $\hat K=1$,  then $\hat g$ is a round metric on $S^2$ and hence there is a conformal transformation $\gamma$ of
$S^2$ for which $\gamma^* \hat g=g_0$.  From Property a) we know  that $\hat g=\exp(2f)g_0$ and so $\gamma^*g_0=\exp(-2f\circ T)g_0.$ This conformal transformation induces an isometry of hyperbolic space which we still denote by $\gamma$.  The relationship between the radial functions $r(x)$ and $s(x)$ is determined by
\begin{equation*}
|s(x)+f\circ \gamma(x)-r\circ\gamma(x)|\leq C\exp(-r(x))
\end{equation*}
for some constant $C$. This implies that for all  $x$ in $\Sigma(q_0)$
$$|w(x)+\hat q_0-q_0|\leq C\exp(-q_0),$$
and thus
$$\lim_{q_0\to\infty}w=0.$$

\end{proof}

\section{Long time existence for inverse mean curvature flow on asymptotically hyperbolic $3$-manifolds}\label{long}

In this section the ambient manifold will be an Anti--de Sitter--Schwarzschild metric $(S^2\times[s_0,+\infty), g_m)$ with mass $m>0$. 

A sphere $\Sigma_0$ satisfies hypothesis $(H)$ if we can find constants $(Q_j)_{j\in\N},\,\varepsilon_0,$ and $\delta_0$ for which
\begin{equation*}
(H)\qquad\left \{ \begin{aligned}
					&|H| \geq \varepsilon_0 \quad\mbox{and}\quad \ar |H^2-4|\leq Q_0,\\
					&\langle\nu,\partial_r \rangle \geq \varepsilon_0\quad\mbox{and} \quad\langle\nu,\partial_r \rangle\geq 1-|\Sigma_0|^{-1}Q_1,\\
					& \n^2 \leq(1/4-\delta_0)H^2\quad\mbox{and}\quad\ar^2\n^2\leq Q_2,\\
					 &\sup_{\Sigma_0}|\nabla^n A|^2\leq Q_{n+2}|\Sigma_0|^{-(n+2)}\quad\mbox{for all } n\geq 1,\\
					 &\mbox{$\Sigma_0$ bounds a compact region containing $S^2\times \{s_0\}$.}
				\end{aligned}
\right.
\end{equation*}
 
 Recall that $\overline r_0$ and $\underline r_0$ denotes, respectively, the outer radius and the inner radius of $\Sigma_0$.
 
\begin{thm}\label{imcf}
Assume that $\Sigma_0$ satisfies $(H)$.

There is a constant $\underline r=\underline r((Q_j)_{j\in\N},\varepsilon_0,\delta_0,\overline r_0-\underline r_0,m)$ such that if $\underline{r}_0\geq\underline r$ then
the inverse mean curvature flow $(\Sigma_t)$ with initial condition $\Sigma_0$ exists for all time and has the
following properties:
\begin{enumerate}
\item [(i)] There is a positive constant $C=C((Q_j)_{j\in\N},\varepsilon_0,\delta_0,\overline r_0-\underline r_0,m)$ such that the mean curvature of $\Sigma_t$ satisfies
$$H\geq C$$
and, for some other constant $C=C((Q_j)_{j\in\N},\varepsilon_0,\delta_0,\overline r_0-\underline r_0,m),$
$$\ar|H^2-4|\leq C\exp(-t);$$
\item[(ii)]
For every $n\geq 0$ and $k\geq 1$ there is a constant 
$$C=C((Q_j)_{j\in\N},\varepsilon_0,\delta_0,\overline r_0-\underline r_0,m)$$
such that
$$\ar^{n+2}|\partial_t^k\, \nabla^n A|^2\leq C \exp(-(n+2)t)$$
and
 $$\ar^{n+2}| \nabla^n A|^2\leq C \exp(-(n+2)t)\quad\mbox{for n} \geq 1;$$

\item [(iii)] The surfaces $\Sigma_t$ can be described as
$$\Sigma_t=\{(\hat r_t+f_t(\theta),\theta)\,|\,\theta \in S^2\},$$
where $\hat r_t$ is such that 
$$|\Sigma_t|=4\pi\sinh^2 \hat r_t.$$
Moreover, the functions $f_t$ converge to a smooth function $f_{\infty}$ defined on $S^2$.
\item [(iv)]For every $n\geq 0$ and $k\geq 1$ there is a constant 
$$C=C((Q_j)_{j\in\N},\varepsilon_0,\delta_0,\overline r_0-\underline r_0,m)$$ such that
$$\ar^n|\nabla^n f_t|^2\leq C\exp(-nt), \quad|\partial_t^k\, f_t|\leq C\exp(-t),$$
and
$$\ar^n|\partial_t^k\, \nabla^n f_t|^2\leq C\exp(-nt)\quad\mbox{for n} \geq 1.$$
\end{enumerate}
\end{thm}
We essentially adapt to our setting some of the ideas used in the work of Huisken--Ilmanen \cite{Huisken4} and Claus Gerhardt \cite{gerhardt} on smooth solutions to inverse mean curvature flow. We could have been more precise regarding how the constants depend on  $(Q_j)_{j\in\N}$ but this version of the theorem suffices for our purposes. The important point is that the estimates do not depend on $\underline r_0$ (only on $\overline r_0-\underline r_0$).
\begin{proof}
During the first part of this proof,  given any geometric quantity $T$ defined on $\Sigma_t$ we use the notation 
$$T=O(\exp(-kr))$$
whenever there is a constant $C=C(m)$ such that
$$|T|\leq C\exp(-kr).$$

Because $H>0$ we have short-time existence for the flow. 
Denoting by $\Sigma^m_t:=\{|x|=r^m_t\}$ the solution to inverse mean curvature flow with initial condition  $\{|x|=t_0\}$ we know that
$$|\Sigma^m_t|=|\Sigma^m_0|\exp(t)$$
and thus we can find a constant $K=K(m)$ such that
$$t/2-K\leq r^m_t-t_0\leq t/2+K.$$ Because two solutions that are initially disjoint must remain disjoint \cite[Theorem 2.2]{Huisken2},  we have that for some constant 
$K=K(m)$
\begin{equation}\label{rosa}
t/2+\underline r_0-K \leq \underline r_t\leq \overline r_t\leq t/2+\overline{r_0} +K\quad\mbox{and}\quad \underline r_t\geq \underline r_0.
\end{equation}
Therefore, we can find $K=K(m,\overline r_0-\underline r_0)$ for which
$$K^{-1}|\Sigma_0|\exp(t)\leq \exp(2r)\leq K|\Sigma_0|\exp(t).$$

We now derive the evolution equations that will be needed later on. We use the
notation
$$B_{ij}\thickapprox C_{ij}$$
when $B_{ij}$ and $C_{ij}$ have the same trace-free part.

Set
$$X:={\phi(r)}\partial_r\quad\mbox{and}\quad\beta_t:=\exp(-t/2)\langle X,\nu\rangle,$$
where the function $\phi$ is such that $g_m=dr^2+\phi(r)^2g_0$ and $\nu$ is the exterior normal vector to $\Sigma_t$. 

\begin{lemm}\label{evol1}
The following evolution equations hold.

\begin{itemize}
	\item [a)]
	\begin{equation*}
		\frac{d\beta_t}{dt}=\frac{\Delta\beta_t}{H^2}+\left(\frac{|A|^2}{H^2}-\frac{1}{2}\right)\beta_t
		+|\partial_r^{\top}|^2\left(\frac{3m}{2\sinh^3 r}+O(\exp(-5r)\right)\frac{\beta_t}{|H|^2};
	\end{equation*}
	\item[b)]$$\frac{d H}{dt} = \frac{\Delta H}{H^2}-(|A|^2+Rc(\nu,\nu))\frac{1}{H}-\frac{2|\nabla H|^2}{|H|^3};$$
	\item[c)]
    	\begin{multline*}
        		\frac{d \tf}{dt}\thickapprox \frac{\Delta \tf}{H^2}-\frac{2\nabla H\otimes\nabla H}{H^3}-\tf-\frac{2\tf^2}{H}\\+
        		\left(\frac{\lvert \tf \rvert
        		^2}{H^2}-\frac{H^2+2Rc(\nu,\nu)+O(\exp(-3r))}{2H^2}\right)\tf_{}\\
        		+\left(\frac{1}{H}+\frac{1}{H^2}\right)O(\exp(-3r));
    	\end{multline*}
	\item[d)]
	\begin{multline*}
		\frac{d\n^2}{dt}\leq\frac{\Delta \n^2}{H^2}+2\left(\frac{\lvert \tf \rvert
		^2}{H^2}-\frac{H^2+2Rc(\nu,\nu)+O(\exp(-3r)}{2H^2}\right)
        		\lvert \tf \rvert ^2\\-2\n^2-2\frac{\lvert\nabla\tf\rvert^2}{H^2}
         	-4\frac{\langle \nabla H \otimes\nabla H, \tf\rangle}{H^3}\\
         	+ \left(\frac{\n}{H^2}+\frac{\n}{H}\right)
         	O(\exp(-3r)).
	\end{multline*}
\end{itemize}

\end{lemm}
\begin{proof}
For every vector $Y$ we have that
    $$D_Y X=\phi'(r)Y$$
and this implies that, using local coordinates $(y_1, y_2)$ for $\Sigma_t$,
$$\langle\nabla\beta_t,\partial_i\rangle =\exp(-t/2)A(\partial_i,X^{\top}),\quad i=1,2$$
and
\begin{multline*}
\exp(t/2)\Delta
\beta_t=\sum_{i}\left(\nabla_{\partial_i}A\right)(\partial_i,X^{\top})+A(\partial_i,\nabla_{\partial_i}X^{\top})\\
=\langle \nabla H, X\rangle +Rc(\nu, X^{\top})+\phi'H-\langle X,\nu\rangle|A|^2.
\end{multline*}
 Moreover
$$D_{\partial_t}\nu=\nabla H/H^2, \quad D_{\partial_t}X=\phi'\partial_t,$$
and so
$$\frac{d\beta_t}{dt}=\frac{\phi'}{H}+\frac{\langle\nabla H, X\rangle}{H^2}-\frac{\beta_t}{2}.$$
Therefore
$$
\frac{d\beta_t}{dt}=\frac{\Delta\beta_t}{H^2}+\left(\frac{|A|^2}{H^2}-\frac{1}{2}\right)\beta
-\exp(-t/2)\frac{Rc(\nu,X^{\top})}{|H|^2}.
$$
Note that denoting by $e_1, e_2$ a $g_m$-orthonormal basis for the coordinates spheres
$$
\langle\nu,X^{\top}\rangle=0\,\Rightarrow\,\langle\nu,\partial_r\rangle\langle\partial_r,X^{\top}\rangle
= -\sum_i\langle\nu,e_i\rangle\langle e_i,X^{\top}\rangle
$$
and hence, we obtain from \cite[Lemma 3.1 (iii)]{tian} that
\begin{multline*}
Rc(\nu,X^{\top})=\sum_i\langle\nu,e_i\rangle
Rc(e_i,X^{\top})+\langle \nu,\partial_r\rangle Rc(\partial_r,X^{\top})\\
=\left(\frac{m}{2\sinh^3 r}+O(\exp(-5r)\right)\sum_i\langle\nu,e_i\rangle \langle e_i,X^{\top}\rangle\\
-\left(\frac{m}{\sinh^3 r}+O(\exp(-5r)\right)\langle\nu,\partial_r\rangle\langle\partial_r,X^{\top}\rangle\\
=|\partial_r^{\top}|^2\left(-\frac{3m}{2\sinh^3 r}+O(\exp(-5r)\right)\langle\nu,X\rangle.
\end{multline*}
The second evolution equation was derived in \cite[Section 1]{Huisken2}.

We now prove the third identity.
 From \cite[Theorem 3.2]{Huisken1} it
follows that
    assuming normal coordinates around a point $p$
    \begin{equation*}
        \frac{d \tf_{ij}}{dt}\thickapprox\frac{dA_{ij}}{dt}-A_{ij}\thickapprox
        \frac{\nabla_i\nabla_j H}{H^2}-\frac{2\nabla_i H\nabla_j H}{H^3}+\frac{\tf_{ik}\tf_{kj}-R_{\nu i\nu j}}{H}.
    \end{equation*}
    Arguing like in the proof of Simons' identity for the Laplacian of the second fundamental form $A$ (see for instance
    \cite{Huisken1}), one can see that
    \begin{multline*}
        \Delta \tf_{ij} \thickapprox \nabla_i\nabla_j H+H\tf_{im}\tf_{mj}+\tf_{ij}H^2/2-\tf_{ij}\n^2+H{R}_{\nu i \nu j}\\
        -{R}_{\nu\nu}\tf_{ij}+{R}_{kikm}\tf_{mj}+{R}_{kjkm}\tf_{im}+{R}_{kijm}\tf_{km}+{R}_{mjik}\tf_{km}\\
        +{D}_k {R}_{\nu j i k}+ {D}_i {R}_{\nu k j k}.
    \end{multline*}
    Because the metric $g_m$ satisfies
    \begin{align*}
        {R}_{stuv} & =-(\delta_{su}\delta_{tv}-\delta_{sv}\delta_{tu})+O(\exp({-3r}))\\
         D_q R_{stuv}&=O(\exp(-3r))
    \end{align*}
    it follows that
        \begin{multline*}
            {R}_{kikm}\tf_{mj}+{R}_{kjkm}\tf_{im}+{R}_{kijm}\tf_{km}+{R}_{mjik}\tf_{km}=-4\tf_{ij}+\tf_{ij}O(\exp(-3r))\\
            =2Rc(\nu,\nu)\tf_{ij}+\tf_{ij}O(\exp(-3r)),
        \end{multline*}
        $${R}_{\nu i \nu j} =-g_{ij}+O(\exp(-3r)),$$
    and therefore
    \begin{multline*}
        \frac{d \tf_{ij}}{dt}\thickapprox \frac{\Delta \tf_{ij}}{H^2}-\frac{2\nabla_i H\nabla_j H}{H^3}\\+
        \left(\frac{\lvert \tf \rvert
        ^2}{H^2}-\frac{H^2+2Rc(\nu,\nu)+O(\exp(-3r))}{2H^2}\right)\tf_{ij}\\
        +\left(\frac{1}{H}+\frac{1}{H^2}\right)O(\exp(-3r)).
    \end{multline*}
    Using the formula
    $$\frac{d\tf}{dt}(\partial_{i},\partial_j)=\frac{d \tf_{ij}}{dt}-\langle D_{\partial_t}\partial_i,\partial_k\rangle\tf_{kj}-\langle D_{\partial_t}\partial_j,\partial_k\rangle\tf_{ik}$$
    we obtain Lemma \ref{evol1} c). The last identity follows from
    $$\frac{d\n^2}{dt}=2\left\langle\frac{d\tf}{dt},\tf\right\rangle$$
    and
    $$\langle\tf^2,\tf\rangle=0.$$
\end{proof}

 We now argue that we can choose $\underline r=\underline r(m)$ and a positive constant $C=C(\varepsilon_0,\overline r_0-\underline r_0, m)$ such that if
  $\underline r_0\geq \hat {r}$, then
\begin{equation}\label{strokes}
    H\geq C\quad\mbox{and}\quad \langle \nu,\partial_r\rangle\geq C\exp(\underline r_0-\overline r_0)
\end{equation}
 while the solution exists.

 Choosing $\underline r$
 large enough so that for all $r\geq \underline r$ the term
 $$\frac{3m}{2\sinh^3 r}+O(\exp(-5r))$$
in the equation of Lemma \ref{evol1} a) is positive, we obtain that
$$\frac{d\beta_t}{dt}\geq\frac{\Delta\beta_t}{H^2}$$
while $\beta_t$ is nonnegative and thus  $\beta_t\geq \min\beta_0>0.$
Note that $\phi(r)$ grows like $\exp(r)$ and so, for some constant $C=C(m)$, 
$$\beta_t\leq C\exp(\overline r_0)\langle\partial_r, \nu\rangle.$$
This implies the desired
bound  for $\langle \nu,\partial_r\rangle$.

Set $\alpha_t:=\beta_t H$. Because
 $$Rc(\nu,\nu)=-2+O(\exp(-3r)),$$
the previous lemma implies that, provided we choose $\underline r$ sufficiently large,

\begin{align*}
	\frac{d\alpha_t}{dt} & =\frac{\Delta \alpha_t}{H^2}-\frac{2\langle\nabla \alpha_t,\nabla
	H\rangle}{H^3}+(4+O(\exp(-3r))-H^2) \frac{\alpha_t}{2H^2}\\
	& \geq \frac{\Delta \alpha_t}{H^2}-\frac{2\langle\nabla \alpha_t,\nabla
	H\rangle}{H^3}+(3-H^2) \frac{\alpha_t}{2H^2}
\end{align*}

 Because  $\alpha_t\leq \sqrt{3}(\min\beta_0)$ implies that
$H^2\leq 3$, it follows from the maximum principle that $\alpha_t\geq \min\{\sqrt{3}(\min\beta_0), \min \alpha_0 \}$ for all $t$ and thus we can use the inequalities in \eqref{rosa} in order to obtain the desired bound for the mean curvature.

\begin{lemm}\label{hearts}We can find constants $\underline r=\underline r(\varepsilon_0,\delta_0,\overline r_0-\underline r_0,m)$ and $C=C(\varepsilon_0,\delta_0,\overline r_0-\underline r_0,m)$ such that if 
$\underline{r}_0\geq\underline r$, then 
$$|\Sigma_0||H^2-4|\leq C\left(|\Sigma_0|\sup_{\Sigma_0}(|H^2-4|+\n^2)+\exp(-\underline r_0)\right)\exp(-t)$$
and
$$|\Sigma_0|^2\n^2\leq C\left(|\Sigma_0|^2\sup_{\Sigma_0}\n^2+\exp(-2\underline r_0)\right)\exp(-2t)$$
while the solution exists.
\end{lemm}
\begin{proof} We assume that the bounds in \eqref{strokes} hold.
    Let $\alpha_t:=\n^2H^{-2}$. From Lemma \ref{evol1} 
    $$\frac{d H^{-2}}{dt} = \frac{\Delta H^{-2}}{H^2}+2\left(\frac{\n^2}{H^2}+\frac{H^2+2Rc(\nu,\nu)}{2H^2}\right)H^{-2}
    -\frac{2|\nabla H|^2}{|H|^6}$$
    and thus
    \begin{equation}\label{montijo}
    \frac{d\alpha_t}{dt}\leq\frac{\Delta \alpha_t}{H^2}+4\alpha_t^2-2\alpha_t+\left(\alpha_t+\sqrt{\alpha_t}
    +H^{-1}\sqrt{\alpha_t}\right)
    \frac{O(\exp(-3r))}{H^2}+Q,
    \end{equation}
    where
    $$Q:=4\frac{\langle\nabla H,\nabla \n^2\rangle}{H^5}-2\n^2\frac{|\nabla H|^2}{H^6}-2\frac{|\nabla \tf|^2}{H^4}-
    4\frac{\langle \nabla H \otimes\nabla H, \tf\rangle}{H^3}.$$
    We claim that, given $\varepsilon>0$, we can find a constant $C=C(\varepsilon_0,\varepsilon,\overline r_0-\underline r_0,m)$ so that
    \begin{equation*}
    \frac{d\alpha_t}{dt}\leq\frac{\Delta \alpha_t}{H^2}+4\alpha_t\left(\alpha_t-\frac{1}{2}+\varepsilon\right)
    +C\exp(-6\underline r_0-3t)+Q
    \end{equation*}
    and
    \begin{equation*}
    Q(p)\leq 4\frac{|\nabla H|^2}{H^4}\left((1+\varepsilon)\alpha_t(p)-\frac{1}{4}+\varepsilon\right)+C\exp(-6\underline r_0-3t),
    \end{equation*}
    whenever $p$ is a critical point of $\alpha_t$.

    The first inequality follows easily from Cauchy's inequalities combined with properties \eqref{rosa} and
    \eqref{strokes}.
     Denote by $\{v_1,v_2\}$ an eigenbasis for $\tf$ at $p$ and assume without loss of generality
     that $\tf(v_1,v_1)\geq 0$. Because $p$ is a critical point  of $\alpha_t$ the following identities hold at $p$
    $$\nabla\n^2=2\n^2 H^{-1}\nabla H$$
    and
    $$\n H^{-1} \nabla H=\sqrt 2\nabla \tf(v_1,v_1)=-\sqrt 2\nabla \tf(v_2,v_2).$$
    As a result, we obtain that
    $$|\nabla \tf|^2=\n^2\frac{|\nabla H|^2}{H^2}+2|\nabla \tf(v_1,v_2)|^2=\alpha_t^2|\nabla H|^2+2|\nabla \tf(v_1,v_2)|^2$$
    and
    \begin{multline*}
        2|\nabla \tf(v_1,v_2)|^2=2|\nabla_{v_2} A(v_1,v_1)+Rc(\nu,v_2)|^2+2|\nabla_{v_1} A(v_2,v_2)+Rc(\nu,v_1)|^2\\
        \leq 2|\nabla_{v_2} A(v_1,v_1)|^2+2|\nabla_{v_1} A(v_2,v_2)|^2\\
        +(\sqrt\alpha_t|\nabla H|+|\nabla H|)O(\exp(-3r))+O(\exp(-6r))\\
        = 2\left|\nabla_{v_2} \tf(v_1,v_1)+\frac{\langle \nabla H,v_2\rangle}{2}\right|^2+
        2\left|\nabla_{v_1} \tf(v_2,v_2)+\frac{\langle \nabla H,v_1\rangle}{2}\right|^2\\
        +(\sqrt\alpha_t|\nabla H|+|\nabla H|)O(\exp(-3r))+O(\exp(-6r))\\
        =2|\langle \nabla H,v_2\rangle|^2\left(\frac{\alpha_t}{\sqrt 2}
        +\frac{1}{2}\right)^2+
        2|\langle \nabla H,v_1\rangle|^2\left(\frac{\alpha_t}{\sqrt 2}-\frac{1}{2}\right)^2\\
        +(\sqrt\alpha_t|\nabla H|+|\nabla H|)O(\exp(-3r))+O(\exp(-6r))\\
        =\alpha_t^2|\nabla H|^2+|\nabla H|^2/2-2\frac{\langle \nabla H \otimes\nabla H, \tf\rangle}{H}\\
         +(\sqrt\alpha_t|\nabla H|+|\nabla H|)O(\exp(-3r))+O(\exp(-6r)).
    \end{multline*}
    Moreover, we also have that at the point $p$
    $$4\langle\nabla H,\nabla \n^2\rangle=8\alpha_t\frac{|\nabla H|^2}{H}$$
    and thus
    \begin{multline*}
        Q(p)=4\frac{|\nabla H|^2}{H^4}\left(\alpha_t(p)-\frac{1}{4}\right)\\+
    (\sqrt\alpha_t|\nabla H|+|\nabla H|)\frac{O(\exp(-3r))}{H^4}+\frac{O(\exp(-6r))}{H^4}.
    \end{multline*}
    The claim follows from Cauchy's inequalities combined with properties \eqref{rosa} and \eqref{strokes}.

	As a result, there is a constant $C_1=C_1(\varepsilon_0,\varepsilon,\overline r_0-\underline r_0,m)$ for which if we set
	$$\beta_t:=\alpha_t+C_1\exp(-6\underline r_0-3t),$$
	then
	 \begin{equation}\label{fogo1}
    \frac{d\beta_t}{dt}\leq\frac{\Delta \beta_t}{H^2}+4\alpha_t\left(\alpha_t-\frac{1}{2}+\varepsilon\right)+Q
    \end{equation}
    and
    \begin{equation*}
    Q(p)\leq 4\frac{|\nabla H|^2}{H^4}\left((1+\varepsilon)\alpha_t(p)-\frac{1}{4}+\varepsilon\right)
    \end{equation*}
    whenever $p$ is a critical point of $\beta_t$.
    
    Chose $\varepsilon<\delta_0/4$ so that
    $$(1+\varepsilon)\left(\frac{1}{4}-\frac{\delta_0}{4}\right)-\frac{1}{4}+\varepsilon \leq 0$$
    and chose $\underline r$ so that  $C_1\exp(-6\underline r) \leq \delta_0/4$. Thus $\beta_0\leq 1/4-3\delta_0/4$ and 
    $$\beta_t(x)\leq 1/4-\delta_0/2 \,\Longrightarrow \,\alpha_t(x)\leq 1/4-\delta_0/4.$$ Therefore we can apply the maximum principle to $\beta_t$ and conclude that  $$\alpha_t \leq \sup \alpha_0+C_1\exp(-6\underline r_0)$$ while the solution exists. This implies that
    $$\alpha_t-\frac{1}{2}+\varepsilon\leq-1/4-\delta_0/8$$ and so we obtain from  equation \eqref{fogo1}  that
    $$\alpha_t\leq(\sup \alpha_0+C\exp(-6\underline r_0))\exp(-(1+2\delta_0)t)$$ 
   for some $C=C(\varepsilon_0,\delta_0,\overline r_0-\underline r_0,m)$. As a result, 
   \begin{align*}
   \alpha_t^2 &\leq C(\sup \alpha^2_0+\exp(-12\underline r_0))\exp(-2t-4\delta_0 t),\\
   \sqrt{\alpha_t}\exp(3r)&\leq C(\sup \alpha_0+\exp(-6\underline r_0))\exp(-2t-\delta_0 t)
   \end{align*}
   for some $C=C(\varepsilon_0,\delta_0,\overline r_0-\underline r_0,m)$. Using this bounds in  equation \eqref{montijo} we obtain
   \begin{multline*}
   	\frac{\alpha_t}{dt}\leq\frac{\Delta \alpha_t}{H^2}-(2+C\exp(-3t/2))\alpha_t\\
	+C({\sup \alpha_0}+\exp(-6\underline r_0))\exp(-(2+\delta_0)t)+C \exp(-6\underline r_0-3t)+Q
   \end{multline*}
 for some $C=C(\varepsilon_0,\delta_0,\overline r_0-\underline r_0,m)$ and hence
    $$\n^2\leq C\left(\sup_{\Sigma_0} \n^2+\exp(-6\underline r_0)\right)\exp(-2t).$$
   
    The evolution equation for $H^2$ is given by (see \cite[Section 1]{Huisken2})
    \begin{equation*}
    \frac{d H^2}{dt} = \frac{\Delta H^2}{H^2}-\frac{6|\nabla H|^2}{|H|^2}-2\n^2-H^2-2Rc(\nu,\nu)
    \end{equation*}
    and thus, if we set $\phi_t:=\exp(t)(H^2-4)$, we obtain that
    \begin{equation*}
    \frac{d \phi_t}{dt} = \frac{\Delta \phi_t}{H^2}-\frac{3\langle \nabla\phi_t,\nabla H\rangle}{H^2}-2\n^2\exp(t)
    -\exp(t)(4+2Rc(\nu,\nu))
    \end{equation*}
    From the upper bound  derived for $\n$ and the bounds given in \eqref{rosa} and
    \eqref{strokes}
     we have that
    \begin{multline*}
    \frac{d \phi_t}{dt} \geq \frac{\Delta \phi_t}{H^2}-\frac{3\langle \nabla\phi_t,\nabla H\rangle}{H^2}\\
    -C\left(\sup_{\Sigma_0} \n^2+\exp(-6\underline r_0)\right)\exp(-t)-C\exp(-t/2-3\underline r_0)
    \end{multline*}
    where $C= C(\varepsilon_0,\delta_0,\overline r_0-\underline r_0,m) $. The maximum principle implies that 
    $$H^2\geq 4-C\left(\sup_{\Sigma_0}(|H^2-4|+\n^2)+\exp(-3\underline r_0)\right)\exp(-t).$$
    
    In order to show the existence of some $C=C(\varepsilon_0,\delta_0,\overline r_0-\underline r_0,m)$ for which
    $$H^2\leq 4+C\left(\sup_{\Sigma_0}(|H^2-4|+\n^2)+\exp(-3\underline r_0)\right)\exp(-t)$$
	it is enough to note that
$$\frac{d H^2}{dt} \leq \frac{\Delta H^2}{H^2}+4-H^2+C\exp(-3\underline r_0-3t/2).$$

\end{proof}

Fix some $\underline r$ for which Lemma \ref{hearts} holds. Note that in this case we
have a uniform bound for $|A|^2$ and so standard estimates can be used to show that the solution
$(\Sigma_t)_{t\geq 0}$ exists for all time. Nonetheless, we need shaper  estimates on all the derivatives of $A$ and this will occupy most of the rest of the proof. What we have done so far proves Theorem \ref{imcf} (i). The next lemma will be useful in proving Theorem \ref{imcf} (iii).

\begin{lemm}\label{so}
There is a constant $C=C(Q_0,\varepsilon_0,\delta_0,\overline r_0-\underline r_0,m)$ such that
$$1-\langle\partial_r,\nu \rangle\leq\ C\left(\sup_{\Sigma_0}(1-\langle\partial_r,\nu\rangle)+\exp(-3\underline r_0)\right)\exp(-t)$$
for all $t$ and thus
$$\ar |\nabla r|^2\leq C\exp(-t)$$
for some other constant $C=C(Q_0, Q_1,\varepsilon_0,\delta_0,\overline r_0-\underline r_0,m).$
\end{lemm}
\begin{proof}
We denote by $\Lambda$ any geometric quantity defined on $\Sigma_t$ for which we can find a constant $C=C(\varepsilon_0,\delta_0,\overline r_0-\underline r_0,m)$ such that
$$|\Lambda |\leq C(|H-2|+\n+\exp(-2r))$$
For every vector $Y$ we have that
    $$D_Y \partial_r={\phi'(r)}/{\phi(r)}\left(Y-\langle Y,\partial_r\rangle\partial_r\right).$$
Therefore
\begin{align*}
\frac{d\langle\partial_r,\nu \rangle}{dt}&=\left\langle
D_{\partial_t}\partial_r,\nu\right\rangle+\left\langle\partial_r, D_{\partial_t}\nu\right\rangle\\
&=\frac{\phi'}{\phi}\frac{1}{H}-\frac{\phi'}{\phi}\frac{\langle\partial_r,\nu
\rangle^2}{H}+\frac{\langle\partial_r,\nabla H \rangle}{H^2}.
\end{align*}
For every tangent vectors $Z$ and $W$ we have
$$\langle\nabla \langle\partial_r,\nu \rangle, Z\rangle=-\phi'/\phi\langle\partial_r,\nu \rangle\langle\partial_r,Z \rangle+
A(Z,\partial_r^{\top})$$ and
$$\langle \nabla_Z\partial_r^{\top}, W\rangle=\phi'/\phi\langle Z,W\rangle-
\phi'/\phi \langle Z,\partial_r\rangle\langle W,\partial_r\rangle-\langle\partial_r,\nu \rangle A(Z,W),$$ where
$\partial_r^{\top}$ denotes the tangential projection of $\partial_r$. These identities combined with Lemma
\ref{hearts} and with
$$\frac{\phi'}{\phi}=1+O(\exp(-2r))$$
imply that
\begin{multline*}
\d \left(-\frac{\phi'}{\phi}\langle\partial_r,\nu
\rangle\partial_r^{\top}\right)=2\left(\frac{\phi'}{\phi}\right)^2\langle\partial_r,\nu \rangle(\dr^2-1)+
\frac{\phi'}{\phi}\langle\partial_r,\nu \rangle^2H\\- \left(\frac{\phi'}{\phi}\right)'\dr^2\langle\partial_r,\nu
\rangle-\frac{\phi'}{\phi}\langle\partial_r,\nu \rangle A(\partial_r^{\top},\partial_r^{\top})\\
 =-2\left(\frac{\phi'}{\phi}\right)^2\langle\partial_r,\nu \rangle+\langle\partial_r,\nu \rangle\dr^2\\ +
\frac{\phi'}{\phi}\langle\partial_r,\nu \rangle^2H+\dr^2\Lambda
\end{multline*}
and
\begin{multline*}
    \d(A(\cdot,\partial_r^{\top}))=\langle\nabla
    H,\partial_r\rangle+Rc(\nu, \partial_r^{\top})+\frac{\phi'}{\phi}H-\frac{\phi'}{\phi}A(\partial_r^{\top},\partial_r^{\top})
    -|A|^2\langle\partial_r,\nu \rangle\\
    =\langle\nabla
    H,\partial_r\rangle+\frac{\phi'}{\phi}H-\dr^2-\frac{H^2}{2}\langle\partial_r,\nu
    \rangle-\n^2\langle\partial_r,\nu \rangle\\
    +\dr^2\Lambda+O(\exp(-3r)).
\end{multline*}
As a result we get
$$\frac{d\langle\partial_r,\nu \rangle}{dt}=\frac{\Delta\langle\partial_r,\nu \rangle}{H^2}+Q,$$
where
\begin{multline*}
H^2Q=2\left(\frac{\phi'}{\phi}\right)^2\langle\partial_r,\nu \rangle-\langle\partial_r,\nu \rangle\dr^2-
2\frac{\phi'}{\phi}\langle\partial_r,\nu \rangle^2H +\dr^2+\frac{H^2}{2}\langle\partial_r,\nu
\rangle\\+\n^2\langle\partial_r,\nu \rangle+\dr^2\Lambda+O(\exp(-3r)).
\end{multline*}
Setting $\alpha_t:=\langle\partial_r,\nu \rangle-1,$ we obtain from \eqref{strokes} and Lemma \ref{hearts} that
\begin{multline*}
Q=2H^{-2}\left(\frac{H^2}{4}-H\frac{\phi'}{\phi}+\left(\frac{\phi'}{\phi}\right)^2\right)+\frac{\alpha_t}{4}(\alpha_t^2-2\alpha_t-4)\\
+\n^2\langle\partial_r,\nu \rangle+\alpha_t \Lambda+O(\exp(-3r))\\
\geq \frac{\alpha_t}{4}(\alpha_t^2-2\alpha_t-4) +\alpha_t  \Lambda+O(\exp(-3r))\\
\geq -\alpha_t(1-\Lambda)+ O(\exp(-3r)),
\end{multline*}
where the last inequality follows from  $0\geq \alpha_t\geq-1$.  There is $$C=C(Q_0,\varepsilon_0,\delta_0,\overline r_0-\underline r_0,m)$$ for which
$$|\Lambda|\leq C\exp(-t)$$
 and hence
$$\frac{d\alpha_t}{dt}\geq\frac{\Delta \alpha_t}{H^2}-(1+C\exp(-t))\alpha_t+C\exp(-3t/2-3\underline r_0)$$
for some other $C=C(Q_0,\varepsilon_0,\delta_0,\overline r_0-\underline r_0,m).$ This equation implies the desired result.
\end{proof}


For the rest of the proof,  $C$ will denote any constant with dependence 
$$C=C((Q_j)_{j\in\N},\varepsilon_0,\delta_0,\overline r_0-\underline r_0,m).$$

Set $\hat r_t$ to be such that $|\Sigma_t|=4\pi\sinh^2 \hat r_t$ and we remark that $\hat r_t-t/2$ is uniformly
bounded. An immediate consequence of the previous lemma is that $\Sigma_t$ can be written as the graph of a
function $f_t$ over the coordinate sphere $\{|x|=\hat r_t\}$ with 
$$|f_t|\leq C\quad\mbox{and}\quad\ar|\nabla f_t|^2\leq C\exp(-t)$$
for some constant $C$. Furthermore, Lemma \ref{hearts} and Proposition \ref{nelsa}  imply the existence of some constant $C$ for which 
$$\ar^2|\nabla ^2 f_t|^2\leq C\exp(-2t).$$
The next lemma is an adaptation of what was done in \cite[Section 6]{gerhardt}.

 Given two tensors $P$ and $S$ we denote by $S\ast T$ any linear combination of tensors formed by contracting
    over $S$ and $T$.

\begin{lemm}\label{national} There is $\underline r=\underline r((Q_j)_{j\in\N},\varepsilon_0,\delta_0,\overline r_0-\underline r_0,m)$ so that if $\underline r_0\geq \underline r$ the following property holds.

For every $n\geq 0$ there is a constant $C$ such that
$$\ar^n| \nabla^n f_t|^2\leq C\exp(-nt)$$
for all $t$. 
Equivalently, for all $n\geq 1$ there is a constant $C$ for which
$$\ar^{n+2}|\nabla^n A|^2\leq C\exp(-(n+2)t).$$
\end{lemm}
\begin{proof}

We start by showing that it is enough to bound $\nabla^n \tf$.
\begin{lemm}\label{nba}
There exists a constant $C$ for which
$$\ar^3|\nabla A|\leq C\ar^3|\nabla \tf|+C\exp(-3t/2)$$
and
$$\ar^4 |\nabla^2 A|\leq C \ar^4 |\nabla^2 \tf|+C\exp(-2t).$$
Moreover, if we can find a constant $E$ for which $$\ar^{k+2}|\nabla^k \tf|^2\leq E\exp(-(k+2)t)\quad\mbox{for all
}k=1,\ldots n-1,$$ then
$$
\ar^{n+3}|\nabla^{n+1} A|^2\leq C_1\ar^{n+3}|\nabla^{n+1} \tf|^2+C_1\exp(-(n+3)t).
$$
for some constant $C_1=C_1(E,(Q_j)_{j\in\N},\varepsilon_0,\delta_0,\overline r_0-\underline r_0,m)$.
\end{lemm}
\begin{proof}
    On each $\Sigma_t$ consider the $1$-form
   $$B(X)=Rc(X,\nu).$$
   First we estimate the derivatives of $B$. In local coordinates $(x_1,x_2)$, $B$ can be written as
   $$B_j=F_j(r,\nabla r), \quad j=1,2,$$
   where $F_{j}(r,q_1,q_2)$ is defined on $\R^3$ and 
   $$|D^k F_j|\leq S_k\exp(-3r) \quad j=1,2$$
   for some constant $S_k$, provided $(q_1,q_2)$ lie on a fixed compact set.

   We denote by $P$ any tensor on $\Sigma_t$ for which $|P|=O(\exp(-3r))$ and by $Q$ any tensor for which $|Q|=O(\exp(-3r))$ and $$\nabla Q=\nabla r \ast P+\nabla^2 r\ast P.$$
     Using this notation we have
   $$\nabla B=\nabla r\ast Q+\nabla^2 r\ast Q$$
   and  so we can estimate 
    $$\ar^3|B|^2\leq C\exp(-3t),\quad\mbox{and}\quad \ar^4|\nabla B|^2\leq C\exp(-4t)$$
    for some constant $C$.

    Let $\{v_1,v_2\}$ be an orthonormal basis for $\Sigma_t$. We know that for every integer $p$
    $$\nabla^p A(v_1,v_2)=\nabla^p \tf(v_1,v_2)$$
    and $$ \nabla^p A(v_1,v_1)-\nabla^p A(v_2,v_2)= \nabla^p \tf(v_1,v_1)-\nabla^p \tf(v_2,v_2).$$
    Moreover, Codazzi equations imply that for $i\neq j$
    $$\nabla^{p}\nabla_{v_i} A(v_j,v_j)=\nabla^{p}\nabla_{v_j} A(v_1,v_2)-\nabla^{p}B_{i}$$
    and thus
    $$|\nabla^{m+1} A|\leq C|\nabla^{m+1} \tf|+C|\nabla^{m}B|$$
    for every integer $m$. This implies the desired result when $n=0,1$.

    To prove the general result we proceed by induction.
    The inductive hypothesis implies that
    $$\ar^k|\nabla ^k r|^2=\ar^k|\nabla^k f_t|^2\leq C_1\exp(-kt)\quad\mbox{for all } k=1,\ldots,n+1$$
    for some $C_1=C_1(E,(Q_j)_{j\in\N},\varepsilon_0,\delta_0,\overline r_0-\underline r_0,m)$ and thus, using  the expression derived for $\nabla B$, we obtain  
    $$\ar^{k+3}|\nabla^k B|^2 \leq C_1 \exp(-(k+3)t)\quad\mbox{for all } k=1,\ldots,n$$
    for some $C_1=C_1(E,(Q_j)_{j\in\N},\varepsilon_0,\delta_0,\overline r_0-\underline r_0,m)$. 
    Hence, the desired result follows.
\end{proof}

       In what follows $L_1$  will denote any tensor that satisfies the following properties. There exists a
    constant $C$ for which
    $$\ar |L_1|\leq C\exp(-t), \quad\ar^3|\nabla L_1|^2\leq C\ar^3|\nabla A|^2+C\exp(-3t),$$
     and if there is a constant $E$ such that 
      $$\ar^{k+2}|\nabla^k A|^2\leq E\exp(-(k+2)t)\quad\mbox{for all }k=1,\ldots
n-1,$$ then $$\ar^{n+2}|\nabla^{n}L_1|^2\leq C_1\ar^{n+2}|\nabla^n A|^2+C_1\exp(-(n+2)t)$$ for some  constant 
$C_1=C_1(E,(Q_j)_{j\in\N},\varepsilon_0,\delta_0,\overline r_0-\underline r_0,m).$ Likewise, $L_0$
will denote any tensor with the same properties of $L_1$ except that we just require $|L_0|$ to be uniformly
bounded. 

We can see from Lemma \ref{evol1} that the evolution equation for $\tf$ can be written as
    $$\frac{d\tf}{dt}\thickapprox\frac{\Delta \tf}{H^2}-\tf+L_1\ast\tf+M+\nabla A\ast\nabla A\ast L_0,$$ where the tensor $M$ stands for the term 
    $$\left(\frac{1}{H}+\frac{1}{H^2}\right)O(\exp(-3r))$$
   that appears on Lemma \ref{evol1} c). The relevant property of $M$ is that
     $$\ar^2|M|^2\leq C\exp(-3t)$$
     and
     $$\ar^3|\nabla M|^2\leq C\exp(-3t)\ar^3|\nabla A|^2+C\exp(-4t)$$ for some constant $C$. If there is a constant $E$ such that for all $t$
      $$\ar^{k+2}|\nabla^k A|^2\leq E\exp(-(k+2)t)\quad\mbox{for all }k=1,\ldots
n-1,$$ then $$\ar^{n+2}|\nabla^{n}M|^2\leq C_1\exp(-3t)\ar^{n+2}|\nabla^n A|^2+C_1\exp(-(n+3)t)$$
for some other constant $C_1=C_1(E,(Q_j)_{j\in\N},\varepsilon_0,\delta_0,\overline r_0-\underline r_0,m)$.  This  follows from the fact that in local coordinates $(x_1,x_2)$
$$M=\left(\frac{1}{H}+\frac{1}{H^2}\right)F(r,\nabla r),$$
where $F(r,q_1,q_2)$ is a matrix-valued function defined on $\R^3$ for which there is a constant $S_k$ such that, provided $(q_1,q_2)$ lie on a fixed compact set,
 $$|D^k F|\leq S_k\exp(-3r).$$

If $K$ denotes the curvature tensor of $\Sigma_t$, then for any tensor $T$ we know that
$$\Delta \nabla T=\nabla \Delta T+K\ast \nabla T+\nabla K\ast T$$
and
\begin{align*}
	\frac{d \nabla T}{dt}&=\nabla \frac{dT}{dt}-\frac{\nabla T}{2}+\nabla T\ast\tf+T\ast \nabla A\\
	&=\nabla \frac{dT}{dt}-\frac{\nabla T}{2}+\nabla T\ast L_1+T\ast \nabla L_0.
\end{align*}
The last identity comes from the fact that, using normal coordinates,
\begin{multline*}
	\frac{d \nabla T}{dt}(\partial_1,\cdots,\partial_{n+1})=\nabla \frac{dT}{dt}(\partial_1,\cdots,\partial_{n+1})-
\nabla T(\partial_1,\cdots,\partial_{n},D_{\partial_t}\partial_{n+1})\\
+(T\ast \nabla A)(\partial_1,\cdots,\partial_{n+1}).
\end{multline*}
Therefore,
\begin{multline*}
	\frac{d\nabla T}{dt}=\frac{\Delta \nabla T}{H^2}-\frac{\nabla T}{2}+\nabla\left(\frac{dT}{dt}-\frac{\Delta T}{H^2}\right)
+T\ast \nabla L_1\\
+\nabla T\ast L_1+\nabla^2 T\ast \nabla L_1+T\ast \nabla L_0.
\end{multline*} 

Proceeding inductively, it can be checked that
\begin{multline*}
\frac{d\nabla^n \tf}{dt}\thickapprox\frac{\Delta\nabla^n \tf}{H^2}-\left(\frac{n}{2}+1\right)\nabla^{n}\tf+\sum_{j=0}^n\nabla^j\tf\ast
\nabla^{n-j}L_1\\
+\sum_{j=0}^{n-1}\nabla^{j+2}\tf\ast \nabla^{n-j}L_1+\nabla^n M+\sum_{j=0}^{n-1}\nabla^j\tf\ast
\nabla^{n-j}L_0\\
+\sum_{j,k,l\geq 0, j+k+l=n}\nabla^{j+1} A\ast \nabla^{k+1} A\ast \nabla^l L_0
\end{multline*}
and thus we can find a constant $C$ for which
\begin{multline*}
\frac{d|\nabla^n \tf|^2}{dt}\leq\frac{\Delta|\nabla^n
\tf|^2}{H^2}-2\frac{|\nabla^{n+1}\tf|^2}{H^2}-(n+2)|\nabla^n
\tf|^2\\
+C\sum_{j=0}^n|\nabla^j\tf||
\nabla^{n-j}L_1||\nabla^n \tf|+C\sum_{j=0}^{n-1}|\nabla^{j+2}\tf||\nabla^{n-j}L_1||\nabla^n \tf|\\
+C|\nabla^n M||\nabla^n \tf|+C \sum_{j,k,l\geq 0, j+k+l=n}|\nabla^{j+1} A||\nabla^{k+1} A||\nabla^l L_0||\nabla^n
\tf|\\
+C\sum_{j=0}^{n-1}|\nabla^j\tf||\nabla^{n-j}L_0||\nabla^n \tf|.
\end{multline*}

We now show the desired bound when $n=1$.  Recall that for some constant $C$ we have (see Lemma \ref{hearts} and Lemma \ref{nba})
$$|\nabla L_0|+|\nabla L_1|+|\nabla A|\leq C(|\nabla \tf|+\exp(-3t/2)\ar^{-3/2}),$$
 $$|\nabla^2 A |\leq C(|\nabla^2 \tf|+\exp(-2t)\ar^{-2}),\quad\mbox{and}\quad\ar^2\n^2\leq C\exp(-2t).$$
In this case, we can find  $\varepsilon>0$ such that 
\begin{multline*}
\frac{d|\nabla \tf|^2}{dt}\leq\frac{\Delta|\nabla \tf|^2}{H^2}-(3-C\exp(-\varepsilon t))|\nabla \tf|^2+C|\nabla
\tf|^4\\+C\exp(-(3+\varepsilon)t)\ar^{-3}.
\end{multline*}
Hence, if we set $$\alpha_t:=\ar^{3}|\nabla
\tf|^2+\exp(-3t),$$ then
$$
\frac{d\alpha_t}{dt}\leq\frac{\Delta\alpha_t}{H^2}-(3-C\exp(-\varepsilon t))\alpha_t+C\alpha_t^2+C\exp(-(3+\varepsilon)t)
$$
for some other constant $C$. Moreover, from Lemma \ref{evol1} d) and Lemma \ref{hearts}, we can  find some positive constant
$$\bar C=\bar C((Q_j)_{j\in\N},\varepsilon_0,\delta_0,\overline r_0-\underline r_0,m).$$
so that
\begin{multline*}
\frac{d|\tf|^2}{dt}\leq\frac{\Delta|\tf|^2}{H^2}+(\bar C\ar^{-1}-1/2)|\nabla \tf|^2+(\bar C\ar^{-1}-1)|\tf|^2\\
+\bar C\exp(-3t)\ar^{-3}.
\end{multline*}
 Choose $\underline r$ so that 
$$\underline r_0\geq \underline r\,\Longrightarrow\, \bar C\ar^{-1}\leq 1/4$$ 

Set
$$\psi_t:=\frac{\log \alpha_t^2}{2}+K\ar^{3}\n^2,$$
where the constant $K$ will be chosen later. Note that
\begin{multline*}
\frac{d\psi_t}{dt}\leq \frac{\Delta \psi_t}{H^2}-(3-C\exp(-\varepsilon_t))+
(C-K/4)\ar^3|\nabla \tf|^2\\
+\frac{|\nabla \log \alpha_t^2|^2}{4}+C\exp(-3t)
\end{multline*} 
for some constant $C$. Choose $K$ such that $K>4C+4$. If $p$ is a maximum of $\psi_t$, then at $p$
$$\frac{|\nabla \log \alpha_t^2|^2}{4}=K^2|\nabla \n^2|^2\leq 2K^2|\nabla \tf|^2\n^2\leq
CK^2\ar^{-2}|\nabla \tf|^2.$$
We can now chose $\underline r$ so that for al $\underline r_0 \geq \underline r$ we have
$$-\ar^{3}|\nabla \tf|^2(p)+\frac{|\nabla \log \alpha_t^2|^2(p)}{4}\leq 0.$$ 

 The maximum principle implies that
$$\psi_t\leq -3t+C$$
for some constant $C$ and so
$$\ar^3|\nabla \tf|^2\leq C\exp(-3t)$$
for some other constant $C$.

 For $n>1$ we argue by induction. Thus, assume that
$$\ar^{k+2}|\nabla^k A|^2\leq C\exp(-(k+2)t)\quad\mbox{ for all }k=1,\ldots n-1$$ for some constant $C$.
Then, we can find another constant $C$ for which
\begin{align*}
|\nabla ^j L_0|^2+|\nabla^j L_1|^2&\leq C\ar^{-j-2}\exp(-(j+2)t)\quad\mbox{if } 1\leq j\leq n-1,\\
|\nabla ^n L_0|^2+|\nabla^n L_1|^2&\leq C|\nabla^n \tf|^2+C\ar^{-n-2}\exp(-(n+2)t),\\
|\nabla^{n+1}A|^2&\leq C|\nabla^{n+1}\tf|^2+C\ar^{-n-3}\exp(-(n+3)t),\\
|\nabla^{n}A|^2&\leq C|\nabla^{n}\tf|^2+C\ar^{-n-2}\exp(-(n+2)t),
\end{align*}
and
$$|\nabla^n M|^{2}\leq C\exp(-3t)|\nabla^{n}\tf|^2+C\ar^{-n-2}\exp(-(n+3)t).$$

Looking at the evolution equation of $|\nabla^n \tf|^2$, we see that we can find $\varepsilon>0$ and a constant $C$ such that
 \begin{multline*}
\frac{d|\nabla^n \tf|^2}{dt}\leq\frac{\Delta|\nabla^n \tf|^2}{H^2}-((n+2)-C\exp(-\varepsilon t))|\nabla^n
\tf|^2\\
+C\ar^{-n-2}\exp(-(n+2+\varepsilon)t)
\end{multline*}
and the maximum principle implies the desired result.
\end{proof}

In what follows, $C$ continues to denote any constant with dependence
$$C=C((Q_j)_{j\in\N},\varepsilon_0,\delta_0,\overline r_0-\underline r_0,m).$$
One immediate consequence of this lemma is that if we denote by $\nabla_0$ the connection determined by  $g_0$ (the round metric on $S^2$), then for every $n\geq 0$
$$|\nabla^n_0 f_t|\leq C$$
for some constant $C$. Moreover,
$$\frac{d \hat r_t}{dt}=\frac{\sinh \hat r_t}{2\cosh \hat r_t}$$
and thus, combining Lemma \ref{hearts} with Lemma \ref{so}, we have
\begin{align*}\left|\frac{d f_t}{dt}\right|&=\left|\frac{1}{H}-\frac{\sinh \hat r_t}{2\cosh \hat r_t}+(\langle\partial_r,\nu\rangle-1)H^{-1}\right| \\
&\leq C (|H-2|+\exp(-2\underline r_0-2t)+|\langle\partial_r,\nu\rangle-1|)\\
&\leq C\left(\sup_{\Sigma_0}(|H^2-4|+\n^2+|\langle\partial_r,\nu\rangle-1|)+\exp(-2\underline r_0)\right)\exp(-2t),
\end{align*}
for some other constant $C$. As a result, we get that the functions $f_t$ converge to a smooth function $f_{\infty}$ on $S^2$ 
 and so this proves Theorem \ref{imcf} (iii).

We will now argue that for all integers $k\geq 1$ and $n\geq 0$ there is a constant $C$ such that
$$\ar^n|\partial_t^k\, \nabla^n f_t|^2\leq C\exp(-nt)\quad\mbox{for n} \geq 1,\quad |\partial_t^k\, f_t|\leq C\exp(-t),$$
and $$\ar^{n+2}|\partial_t^k\, \nabla^n A|^2\leq C \exp(-(n+2)t).$$
This estimates finish the proof of the theorem.

We start with the case $k=1$.
Using normal coordinates, we have that
\begin{align*}
\langle \partial_t \nabla f_t, \partial_i\rangle &= \partial_t(\partial_i f_t)-\langle\nabla f_t, D_{\partial_t}\partial_i \rangle\\
&=\partial_i(\langle\partial_r,\nu\rangle)H^{-1}-\langle\partial_r,\nu\rangle \langle \nabla H, \partial_i\rangle H^{-2}-A(\nabla f_t, \partial_i)H^{-1}\\
&= -\frac{\phi'}{\phi}\langle\partial_r,\nu\rangle \langle \nabla f_t,\partial_i\rangle H^{-1}-\langle\partial_r,\nu\rangle \langle \nabla H, \partial_i\rangle H^{-2}
\end{align*}
and this implies that
$$\ar|\partial_t \, \nabla f_t|^2\leq C\exp(-t).$$
The same type of computations shows that for every $n\geq 1$ we can find $C$ such that
$$\ar^n|\partial_t \, \nabla^n f_t|^2\leq C\exp(-nt)$$
This implies that, for each $n\geq 1$,
$$\ar^{n+2}|\partial_t\, \nabla^n A|^2\leq C \exp(-(n+2)t)$$
for some constant $C$.
Having this estimates one can then show that
$$|\partial_t^2\, f_t|\leq C\exp(-t)$$
and, for each $n\geq 1$, 
$$\ar^n|\partial_t^2\, \nabla^n f_t|^2\leq C\exp(-nt).$$
Repeating this process gives the desired estimates.

\end{proof}

\section{A modified Shi-Tam flow}\label{stf}

In this section $\Sigma_0$  denotes a sphere satisfying hypothesis (H) and $(\Sigma_t)_{t\geq 0}$ is a solution to \imc flow for which Theorem \ref{imcf} holds.  Consider the manifold $$N:=\bigcup_{t\geq 0}\Sigma_t$$ where the metric $g_m$ can be written as
$$g_m=\frac{dt^2}{H^2}+g_t.$$
The metric $\bar g$ is defined to be
$$\bar g:=\frac{u^2}{H^2}dt^2+g_t,$$
where function $u$ satisfies \eqref{equacao}.

\begin{lemm}\label{yo}
  The metric $\bar g$ has $R(\bar g)=-6.$ 
\end{lemm}
\begin{proof}
    The mean curvature and the exterior normal vector of $\Sigma_t$ computed with respect to $\bar g$ equal
    $$\bar H(\Sigma_t)=H(\Sigma_t)/u\quad\mbox{and}\quad\bar\nu=\nu/u$$
    respectively. Thus $$\frac{\bar \nu}{\bar H}=\frac{\nu}{H}$$
    and this implies that $(\Sigma_t)_{t\geq 0}$ is indeed a solution to \imc flow for the new metric with
    $\bar H(\Sigma_0)=2.$

    We now check that the scalar curvature of $\bar g$ is $-6$.  According to formula $(1.10)$ of \cite{shi-tam}, given metrics
    $$h_0:=dt^2+g_{t}\quad\mbox{and}\quad h_1=v^2dt^2+g_t,$$
     the scalar curvature $R^0$ of $h_0$ and $R^1$ of $h_1$ are related by
    \begin{equation}\label{tom}
    H^0\frac{\partial v}{\partial t}=v^2\Delta_t v+\frac{1}{2}(v-v^3)R_t-\frac{1}{2}uR^0+\frac{u^3}{2}R^1,
     \end{equation}
    where $H^0$ denotes the mean curvature of $\Sigma_t$ with respect to $h_0$.

    Let $g_0$ be the metric $dt^2+g_t$. Because the scalar
    curvature of $g_m$ is $-6$, we obtain from   combining \eqref{tom} (setting $v=H^{-1}$) both with Gauss equations and with
    $$\frac{dH^{-1}}{dt}=\frac{\Delta H^{-1}}{H^2}+\frac{|A|^2+Rc(\nu,\nu)}{H^{3}}$$
   that the scalar curvature of $g_0$ is given by
    $$R(g_0)=R_t-1-\frac{|A|^2}{H^2}.$$
    Consider the function $v:=u/H$.  Using \eqref{tom} with $h_0=g_0$ and $h_1=\bar g$,  the condition that $R(\bar
    g)=-6$ is equivalent to
    $$\frac{\partial v}{\partial t}=v^2\Delta_t v+\frac{1}{2}(v-v^3)R_t-\frac{1}{2}vR(g_0)-3v^3.$$  The
    evolution equation for $u$ follows from the above equation, Gauss equations, and the evolution equation for $H^{-1}.$
  
\end{proof}
Using the identification of $\Sigma_t$ with $S^2$ via
$$\Sigma_t=\{(\hat r_t+f_t(\theta),\theta)\,|\,\theta \in S^2\},$$
the function $u_t$ can be identified with a function on $S^2$ which we still denote by $u_t$.  Recall that the normalized metrics $\hat g_t$ (defined on Lemma \ref{mao}) converge to a
 smooth metric  on $S^2$. The main purpose of this section is to prove

\begin{thm}\label{shiflow}Assume that $\Sigma_0$ satisfies $H(\Sigma_0)>0$ and that on $\Sigma_t$ we have
$$R_t+6-2H\Delta_t H^{-1}>0$$
for all $t$.

Equation
\eqref{equacao} admits a smooth solution $u$ with initial condition $u_{\Sigma_0}=H(\Sigma_0)/2$ and satisfying
the following properties.
\begin{enumerate}
\item[(i)]If we denote by $u_t$ the restriction of $u$ to $\Sigma_t$, then the functions
$$w_t:=2\exp(3t/2)|\Sigma_0|(u_t-1)/(4\pi)$$
converge smoothly to a  function $w_{\infty}$ defined on $S^2$.

\item[(ii)] For every integer $n$ and $k$ we can find
$$\Lambda=\Lambda((Q_j)_{j\in\N},\varepsilon_0,\delta_0,\overline r_0,\underline r_0,m)$$
such that
$$|\nabla ^n w_t|^2\leq \Lambda\exp(-nt)\quad\mbox{and}\quad |\partial_t^k \nabla^n w_t|\leq \Lambda\exp(-(n+2)t);$$

 \item[(iii)] The metric $\bar g$ is asymptotically hyperbolic. More precisely, we can find a coordinate system  $(s,\theta)$ and a symmetric 2-tensor $Q$ such that
\begin{align*}\bar g & =ds^2+\sinh^2 s g_0+\left(\frac{m+(|\Sigma_0|/(4\pi))^{1/2}\exp(3f_{\infty})w_{\infty}}{3\sinh s}\right)g_0+Q
\end{align*}
and
 $$\lvert Q\rvert+\lvert D Q\rvert +\lvert D^2 Q\rvert+\lvert D^3 Q\rvert\leq \Lambda\exp(-4r) $$
for some $\Lambda=\Lambda((Q_j)_{j\in\N},\varepsilon_0,\delta_0,\overline r_0,\underline r_0,m)$.
\end{enumerate}
\end{thm}

Except for property $(iii)$, this theorem was essentially proven in \cite[Theorem 2.1]{mutao} when the deformation vector of
the foliation $(\Sigma_t)_{t\geq 0}$ equals the unit normal vector. In light of Theorem \ref{imcf} the same techniques
apply with no modification (see also \cite{shi-tam}). Nonetheless, we need to make sure that some estimates are independent of $\underline r_0$ and so we  sketch its proof. During the proof 
 $\Lambda$ will denote any constant with dependence
$$\Lambda=\Lambda((Q_j)_{j\in\N},\varepsilon_0,\delta_0,\overline r_0,\underline r_0,m).$$

\begin{proof}
Set
$$h^{+}(t)=\sup_{\Sigma_t}\left(\frac{R_t+6-2H\Delta_t H^{-1}}{2H^2}\right)$$
and $h^{-}(t)=h^{+}(t)$ if $\inf u_0\leq1$ or, in case $\inf u_0>1,$
$$h^{-}(t)=\inf_{\Sigma_t}\left(\frac{R_t+6-2H\Delta_t H^{-1}}{2H^2}\right).$$
Moreover, define 
$$W^{+}=1-\left(\sup_{\Sigma_0} u_0\right)^{-2},\quad W^{-}=1-\left(\inf_{\Sigma_0} u_0\right)^{-2},$$
and
$$\gamma^{\pm}(t)=\left (1-W^{\pm}\exp\left(-\int_0^t 2h^{\pm}(s) ds \right)\right)^{-1/2}.$$
 From  \cite[Lemma 2.2]{shi-tam} (see also \cite[Section 2.2]{mutao}) we have that comparison with the ODE 
$$\frac{d\gamma}{dt}=h^{\pm}(t)(\gamma-\gamma^3),$$
 implies \begin{equation}\label{power}
 \gamma^- (t)\leq u_t\leq\gamma^+(t)
 \end{equation}
  while the solution exists.  Moreover, we know from 
 from Theorem \ref{imcf} that
$$|h^{\pm}(t)-3/4|\leq \Lambda\exp(-t)\quad\mbox{and}\quad |w_0|\leq \Lambda.$$
for some constant $\Lambda$. Therefore, the inequalities in \eqref{power} imply that, while the solution exists, 
$$|w_t|\leq \Lambda$$ 
 for some other constant $\Lambda$.

Performing the change of variable
    $$s=-4\pi|\Sigma_t|^{-1}=-4\pi|\Sigma_0|^{-1}\exp(-t),$$
 the evolution equation for $w_t$ becomes (see also \cite[Theorem 2.1]{mutao})
\begin{multline}\label{portis}
\frac{d w_s}{ds}=\frac{u^2}{H^2}\widehat \Delta_t w_s+2u^2H^{-1}\hat g_t\left(\widehat \nabla w_t,\widehat
\nabla H^{-1}\right)\\+ w_s (4\pi)^{-1}|\Sigma_t|\left(\frac{3}{2}+u(u+1)\left(\frac{\Delta_t
H^{-1}}{H}-\frac{R_t+6}{2H^2}\right)\right),
\end{multline}
where the operators $\widehat \Delta_t$ and $\widehat \nabla$ are computed with respect to the normalized metric $\hat g_t$ and the range os $s$ is
$-4\pi\ar^{-1}\leq s<0.$

In order to use the standard theory for quasilienar parabolic equations, we need to make some remarks regarding the last term on the right-hand side of equation \eqref{portis}. Direct computation shows that
$$u(u+1)=2+3\sqrt\frac{\ar}{16\pi}(-s)^{3/2}w_s-\frac{\ar}{16\pi}s^3w_s^2.$$
Thus the term
\begin{equation}\label{sr}
(4\pi)^{-1}|\Sigma_t|\left(\frac{3}{2}+u(u+1)\left(\frac{\Delta_t
H^{-1}}{H}-\frac{R_t+6}{2H^2}\right)\right)
\end{equation}
can be decomposed as
$$-\frac{3|\Sigma_0|}{164\pi}s^2w^2_s-9\sqrt{\frac{|\Sigma_0|}{164\pi}}\sqrt {-s} w_s+u(u+1)F_t,$$
where
$$F_t=(4\pi)^{-1}|\Sigma_t|\left(\frac{\Delta_t H^{-1}}{H}+\frac{6(H^2-4)-4R_t}{8H^2}\right).$$
Therefore, we obtain from Theorem \ref{imcf} that the term in \eqref{sr} is bounded by some constant $\Lambda$.

Standard theory for quasilinear parabolic
equations \cite[Section VI, Theorem 6.33]{lieb} gives a uniform $C^{0,\alpha}$-bound in space-time for $w_s$, i.e., for all $\theta, \theta' \in S^2$ and $ -4\pi\ar^{-1}\leq s,s'<0$
$$\frac{|w_s(\theta)-w_s(\theta')|}{\mbox{dist}(\theta,\theta')^{2\alpha}}+\frac{|w_s(\theta)-w_{s'}(\theta)|}{|s-s'|^{\alpha}}\leq \Lambda$$
for some constant $\Lambda$. 

The term in \eqref{sr} has a uniform $C^{0,\alpha}$-bound and so 
standard Schauder estimates imply that $\widehat \nabla w_s$ and  $\widehat \nabla^2 w_s$ are uniformly $C^{0,\alpha}$-bounded in space-time.
Bootstrapping implies the existence of a solution $w_s$ for all $s$ with   
$$|\widehat \nabla ^n w_s|+|\partial_s \widehat \nabla^n w_s|\leq \Lambda$$  
for every integer $n$. Rewriting the equation for $w_t$ in terms of the variable $t$ and differentiating it with respect to time we obtain that, for every integer $n$ and $k$,
$$|\nabla ^n w_t|^2\leq \Lambda \exp(-nt)\quad\mbox{and}\quad |\partial_t^k \nabla^n w_t|^2\leq \Lambda\exp(-(n+2)t).$$
As a result,  $w_t$ converges smoothly to a smooth function $w_{\infty}$ defined on $S^2$. 

Finally, we show that the metric $\bar g$ satisfies the definition of asymptotic hyperbolicity given in the
Introduction. The manifold $N$ defined in the beginning of this section is diffeomorphic to
$S^2\times [0,+\infty)$ and thus, besides  polar coordinates $(r,\theta)$, admits also coordinates $(t,\theta)$ where $r=f_t+\hat r_t.$
In what follows we will use these coordinate systems, Theorem
\ref{imcf}, and the previous estimates for the function $u$ without further mention. Let
$$h:=2w_{\infty}\exp(3f_{\infty})|\Sigma_0|^{3/2}(4\pi)^{-3/2}$$
 and denote by $Q$ any $2$-tensor that satisfies
 $$\lvert Q\rvert+\lvert D Q\rvert +\lvert D^2 Q\rvert+\lvert D^3 Q\rvert=O(\exp(-4r)). $$

Then
\begin{align*}
    \bar g&=g_m+\frac{u^2-1}{H^2}dt^2=g_m+\frac{u-1}{2}dt^2+Q\\
    &=g_m+2(u-1)dr^2+Q.
\end{align*}
Due to the fact that
$$|\Sigma_t|=4\pi\sinh^2(r-f_t),$$
we get that
\begin{align*}
    \bar g&=g_m+\frac{2(u-1)\exp(3t/2)|\Sigma_0|^{3/2}}{(4\pi)^{3/2}\sinh^3(r-f_t)}dr^2
    +Q\\
    &=g_m+\frac{16h}{\exp(3r)}dr^2+Q\\
    &=\left(1+4h\exp(-3r)\right)^2dr^2+(\sinh^2 r+m/(3\sinh r))g_0+Q.
\end{align*}
Thus, if we set
$$s:=r-4/3h\exp(-3r),$$
we obtain that
\begin{align*}\bar g & =ds^2+(\sinh^2 s+(h+m)/(3\sinh s))g_0+Q
\end{align*}
and this implies that $\bar g$ is asymptotic hyperbolic if one uses the coordinate system $(s,\theta)$.
\end{proof}

\newpage

\bibliographystyle{amsbook}

\vspace{20mm}

\end{document}